\documentclass[11pt]{amsart}
\usepackage{paralist}
\usepackage{amsmath}
\usepackage{graphics}
\usepackage{epsfig}
 \usepackage[colorlinks=true]{hyperref}
\hypersetup{urlcolor=blue, citecolor=red}

\newtheorem{theo}{Theorem}[section]
\newtheorem{cor}{Corollary}

\newtheorem{lem}[theo]{Lemma}

\theoremstyle{definition}
\newtheorem{definition}[theo]{Definition}
\newtheorem{remark}{Remark}

\newcommand{\eps}[1]{{#1}_{\varepsilon}}

\def\O{{\Omega}}
\def\nl{{\nabla^{l}}}
\def\nr{{\nabla^{r}}}
\def\o{{\omega}}

\def\eps{{\epsilon}}
\def\G{{\Gamma}}

\def\p{{\mathcal{P}}}
\def\f{{\mathcal{F}}}

\def\L{{\mathcal{L}}}

\def\grad{\nabla}

\def\R{{\mathbb{R}}}

\def\Z{{\mathbb{Z}}}
\def\N{{\mathbb{N}}}

\def\P{{\mathbb{P}}}
\def\E{{\mathbb{E}}}

\newcommand{\dem}[1]{\vskip 0.2\baselineskip \noindent {\bf{#1}}\vskip 0.2\baselineskip }
\newcommand{\fdem}{\vskip 0.2 pt \qquad \qquad \qquad \qquad \qquad \qquad \qquad \qquad \qquad \qquad \qquad \qquad \qquad \qquad \qquad \qquad  $\square$  }

\def\tilde{\widetilde}

\thanks{The research was partially supported by the DFG Forschergruppe 718 "Stochastics and Analysis in Complex Physical Systems."}
\begin{document}

%%%%%%%
%\def\cg#1{\textcolor{green}{#1}}

\title[Random Coefficients]
{Non-existence of positive stationary  solutions for a class of semi-linear PDEs with random coefficients}

\author{J. Coville, N. Dirr and S. Luckhaus}

\keywords{Qualitative behavior of parabolic PDEs with random coefficients, 
Random obstacles, Interface Evolution in Random media} 
\subjclass{35R60, 35B09, 82C44}

\email{{\tt jerome.coville@avignon.inra.fr}}
\email{ {\tt N.Dirr@maths.bath.ac.uk}}
\email{{\tt luckhaus@mis.mpg.de}}

%%%%%%%

\maketitle
\centerline{\scshape   J\'er\^ome Coville}
{\footnotesize 
\centerline{Equipe BIOSP,
INRA Avignon}
\centerline{Domaine Saint Paul, Site Agroparc }
\centerline{84914 Avignon cedex 9, France }
}\medskip
\centerline{\scshape Nicolas Dirr}
{\footnotesize \centerline{Department of Mathematical Sciences, 
University of Bath,}
 \centerline{ Bath, BA2 7AY, United Kingdom.}
}\medskip
\centerline{\scshape Stephan Luckhaus}
{\footnotesize
\centerline{Mathematisches Institut der Universit\"at Leipzig,}
\centerline{PF 100920,
Leipzig, Germany} 
}

%\bigskip
% The name of the associate editor will be entered by an editorial staff
%\centerline{(Communicated by the associate editor name)}

\begin{abstract}
We consider a so-called random 
obstacle model for the motion of a hypersurface through a field of 
random obstacles,
driven by a constant driving field. 
The resulting semi-linear parabolic PDE with random coefficients does not
admit a global nonnegative stationary solution, 
which implies that an interface that was flat originally cannot get
stationary. 
The absence of global stationary solutions is shown by 
proving lower bounds on the growth of stationary solutions on 
large domains with Dirichlet boundary conditions. 
Difficulties arise because the random
lower order part of the equation cannot be bounded uniformly.
\end{abstract}

\section{Introduction}  

We are interested in the behavior of a moving interface $\G$  in a random medium, where $\G$ is a graph,  
i.e. defined as
\begin{equation}\label{cdl.interface}
\Gamma(t):=\{ (x,y)\in \R^{2}: y=u(x,t)\}
\end{equation}
and
the function $u$ evolves according to the following equation:
\begin{align}
&\frac{\partial u}{\partial t}= u_{xx}(x,t) + f(x,u(x,t)) +F \quad \text{ in } \quad \R\times \R^{+}, \label{cdl.eq.interface}\\
&u(x,0)=0\label{cdl.eq.initial}
\end{align} 
where $f \in C^1(\R^{2}\times \O)$ is a random field  representing  the  random medium 
and will be defined more precisely later on. Note that $f$ is not restricted to be either positive or negative.
$F$ is a positive constant called "driving field."
The objective is to prove that the solution of \eqref{cdl.eq.interface}-\eqref{cdl.eq.initial} does not get pinned,
i.e. does not converge to a nonnegative stationary solution if
$F$ is above a critical value $F_c$. To this end, we will show that nonnegative stationary solutions on bounded
intervals $[-N, N]$ with Dirichlet boundary conditions get large with high probability as  $N\to \infty.$

The main contribution of this paper is to show that a {\em finite} $F$ is sufficient to keep the graph moving, even if
it will have to pass through regions where $f(x,u,\omega)\ll -1,$ provided the probability of finding such a region is
small. As $f$ can become arbitrarily big, one cannot find a deterministic subsolution that keeps moving,
and instead probabilistic arguments are needed.

The interest in the model stems from
the theoretical analysis of the effective behavior on large scales of
models for interface evolution 
specified at a microscopic scale, which is at the heart of many
problems in physics and material science. Of particular interest is the influence of material heterogeneities, which are generally assumed to be random. Mathematically, this leads to studying the limit of evolution equations with rapidly varying random coefficients. In the case of dissipative equations,
on which we focus here, the randomness leads to new and interesting effects absent in the case of periodic coefficients, e.g. pinning and de-pinning for obstacles with a strength that cannot be bounded uniformly. If the strong obstacles
are sufficiently rare, than the interaction through the Laplacian helps the graph overcome them although the
total forcing $f(x,u)+F$ remains negative near the obstacle.

One example we have in mind as
motivation are driven elastic systems, for a review of the  research
in physics and its possible applications we refer to \cite{BN04}.
For a survey of front evolutions in random media, with evolution 
laws different from the ones considered here,  see e.g. the recent
monograph \cite{Xin}.

The model (\ref{cdl.eq.interface}) is obviously a  gradient flow
for a random energy. In fact, it approximates a more geometric interface
evolution law:

In fact, if the hypersurface $\Sigma$ is the boundary of the set 
$A_\Sigma$ then we can define for
any  bounded $D\subseteq \R^{2}$ the energy
$$
F(\Sigma| D):=H^{1}(\Sigma \cap D)+\int_{D\cap A_\Sigma} f(X,\omega)dX
$$where $X\in \R^{2}$  and $H^{1}$ denotes the 
$1$-dimensional Hausdorff measure.

Requiring that the first variation of 
that functional (with respect to inner variations, i.e. deforming the interface
with the flow of a smooth vector field) is proportional to the 
normal velocity of the interface leads
to forced mean curvature flow,
$$
V=\kappa+f(X),
$$where $\kappa$ denotes the mean curvature of the interface (trace
of the second fundamental form) and the scalar $V$ is the velocity of
the interface in the direction of the inner normal. This geometric evolution law leads to nonlinear degenerate parabolic
equations, hence questions concerning the large-scale behaviour of solutions are related to homogenising such
equations with periodic or random coefficients. 
This is an active field of research (see e.g. \cite{CSW}, \cite{LS}) but many difficult problems
remain open. Here we consider a modified evolution law:

If we suppose that the interface is a graph  is ``flat'' 
(no overhangs, small gradients)
then we can consider  a semi-linear equation as in
(\ref{cdl.eq.interface})
as heuristic approximation of the evolution by forced mean curvature flow.

This model, here called random obstacle model (ROM) because of the precise nature of the random nonlinearity 
$f(x,u,\omega)$ used in this paper, is a special case of a class of
equations sometimes called quenched  Edwards-Wilkinson model 
which,  for some choices of the random nonlinearity,  
is used in physics as a model for 
overdamped interface evolution in 
a random environment when
``overhangs'' can be neglected. 
For further comments on physical properties and justifications
of the model we refer to \cite{BN04}. In particular, one expects that solutions move with a deterministic effective (large-scale) velocity
for $F$ larger than a critical forcing $F_*.$ For $F$ slightly larger than $F_*,$ 
the relation between the effective velocity and $F-F_*$ is expected to be a power law.
(See also \cite{DY} for the periodic case.).

While there are important differences between the forced mean curvature flow and the semi-linear model (e.g. forced
mean curvature flow can "wrap around" strong obstacles),
we expect that the techniques we will develop 
when studying  (\ref{cdl.eq.interface}) will prove helpful
in investigating  more general models for interface evolution. This strategy was successful in the periodic case,
where first the semi-linear case was solved (\cite{DY}) and then the results could be extended to graphs
evolving by forced mean curvature flow (\cite{DKY}).

One more reason why such models are of mathematical interest is the relation with "singular" homogenization problems,
i.e. problems where the $\eps$-equation is of second order (possibly degenerate) and the homogenized equation
of first order.  Note that the effective
velocity $c(\eta)$ of an interface evolving with average slope 
$\eta$ can be found by considering
$$
\frac{\partial u}{\partial t}= u_{xx}(x,t) + f(x,\eta \cdot x+u(x,t)) +F, 
$$ i.e. this can be seen as  the ``cell problem'' for
$$
\frac{\partial v(y,\tau,\omega)}{\partial\tau}={ \eps} v_{yy}(y,\tau,\omega)+f({\eps^{-1}}y, 
{\eps^{-1} }v (y,\tau,\omega),\omega)+F
$$ with $\tau=\eps^{-1}t,\ y=\eps^{-1}x.$

The paper is organised as follows.
In Section 2 we define the random obstacle model precisely and state our main results.

%In section 3 we compare with another model, which is more suitable for explicit estimates because
%the nonlinearity is piecewise constant. On obstacles, $f(x,u)$ equals the negative of a Poisson distributed
%random variable, while elsewhere $f(x,u)=0,$ so stationary solutions solve $u_{xx}=-F$ away from the obstacles.
In Section 3, we introduce an auxiliary model which is more suitable
for explicit estimates and whose solutions can be related  
to solutions of the original  equation \eqref{cdl.eq.interface}
by the comparison principle for parabolic equation. 
This auxiliary problem has the property that   any of its stationary solutions $u$ solve $u_{xx}=-F$ away from the obstacles and is a convex function  on the obstacles.
This fact allows us to define a discretization, using that
each solution is determined by its values when entering and leaving an obstacle.
This yields a discretised path $\bar v^\delta:\ \Z\to \delta\Z$ characterizing each stationary solution.

In section 4, we estimate the discrete Laplacian of $\bar v^\delta(i)$ against the obstacles that sit above and below
$i\in \Z$ and are approached by the path, i.e. $\Delta_d\bar v(i)+\bar F\le C\ell_{i,[\bar v^\delta(i)[}(\omega)$ where $\bar F$ is a constant which can be chosen arbitrarily large.
A technical problem is posed by the fact that the path may pass more 
than one  obstacle above the same integer.

In section 5 we estimate the probability of a discrete  path being "compatible" with the random environment.
This probability can be estimated against 
an auxiliary random measure on paths:
\begin{eqnarray*}
\P\left(\left\{
\omega:\ u(\omega)\  {\rm compatible \ with\  } \bar v^\delta (i)\right\}\right)
&\le& C^{2N}{\P}\left(\{ \Delta_d \bar v^\delta(i) \}_{i=-N+1}^{N-1}\right),
\\
 \P\left( \{\Delta_d \bar v^\delta(i) \}_{i=-N+1}^{N-1}\right)&:=&Z^{-1}
e^{-\lambda\sum_{i=-N+1}^{N-1} |\Delta_{d} \bar v^\delta (i)+\bar F |},
\end{eqnarray*}
where $Z$ is a normalization (corresponding to the partition function in statistical mechanics).
%and $\Delta_d \bar v^\delta(i)\ge -F.$ 

In section 6 we conclude that the probability of a nonnegative solution of the Dirichlet problem to cross
$KN-K|x|$ is ${\mathcal O}(e^{-CN}).$ 
The key observation is that for such a path $N^{-1}\sum_{i=-N+1}^{N-1}\left(\Delta_{d}\bar v^{\delta}(i)+\bar F\right)$ must be large,
which is very unlikely under the auxiliary (product) probability measure.

Finally, we show by invoking the comparison principle for semi-linear parabolic equations 
that these results for large $N$ imply non-existence  of global nonnegative 
stationary solutions. This  implies that for a solution $u$ 
of (\ref{cdl.eq.interface}), (\ref{cdl.eq.initial}) and 
all $x\in \R$ it holds that $\lim_{t\to\infty}u(t,x,\omega)=+\infty$ almost
surely in $\omega,$ i.e. the interface cannot be stopped by the obstacles.

{\bf Acknowledgements} The second named author would like to thank Enza Orlandi and Michael Scheutzow
for helpful discussions. 
The authors acknowledge gratefully the hospitality of the Max Planck Institute for Mathematics in the Sciences Leipzig.  

\section{Results and Definitions}~
\subsection{The random field $f$}
Here, the  field $f$ is negative on "obstacles" in $\R^2$ which are random in strength, but positioned 
on a lattice.  More precisely, we make the following assumption:
\begin{definition}[Obstacles]\label{Obstacles}\hfill
\begin{enumerate}
\item Let $\Z^*:=\Z+1/2.$ We assume that the obstacles lie on a lattice   $\L:=\Z\times\Z^*$ where for convenience $(b_{ij})_{_{i\in \Z,j\in\Z^*}}$  denotes the nodes of this lattice, i.e $b_{i,j}:=(i,j)$.
\item Let $\delta\ll1/2$ and define $Q_\delta(0,0):=[-\delta,\delta]^2,$ and   $Q_{\delta}(i,j):=Q_\delta(0,0)+b_{i,j}.$
Then the obstacles, i.e. regions where  $f<0$ is possible, are given by the $Q_\delta(i,j),$ see also Figure 1
\end{enumerate}
\end{definition}

%\textcolor{red}{Put a picture or a drawing of the lattice with the obstacle}.
\begin{figure}
\label{fig1}
\begin{center}
\input{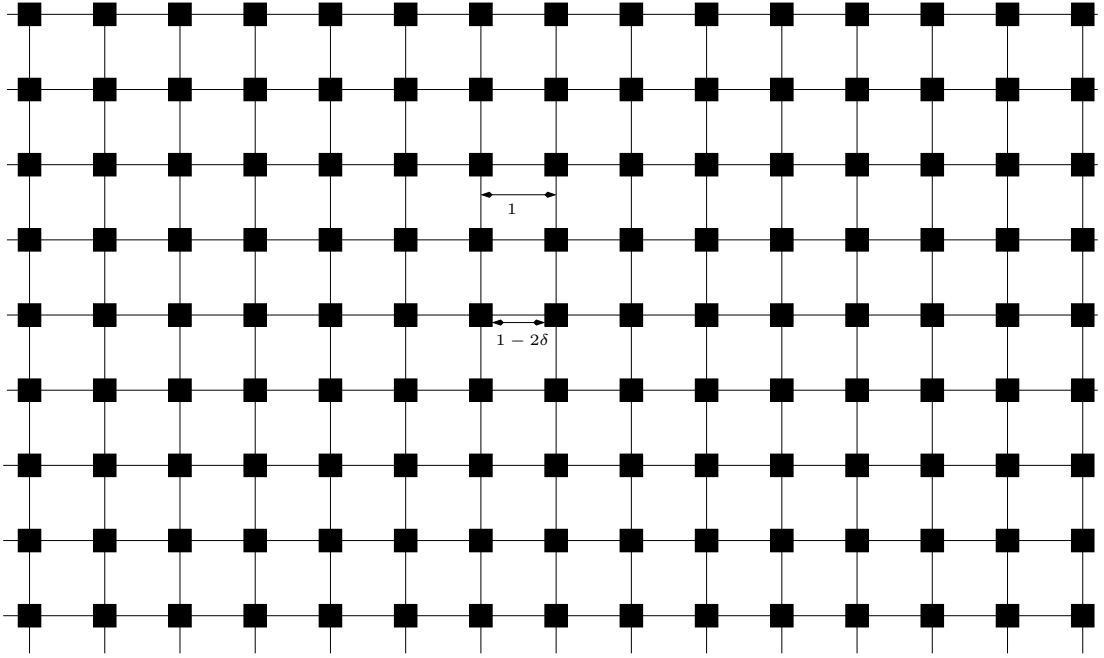}
\caption{The obstacles}
\end{center}
\end{figure}
In order to obtain existence and regularity of the solutions, the nonlinearity $f(x,y)$ should be sufficiently
regular, hence in order to define $f$ we have to smooth out the obstacles.
\begin{definition}[Random field]
Let $\phi \in C^{\infty}_c$ be a nonnegative function %of mass one
such that its  support is contained in  cube $Q_\delta(0,0)$.

Let ($l(i,j)(\o))_{(i,j)\in \Z\times\Z^*}$ be a family of  independent identically distributed exponential 
random variables, i.e.  there exists $\lambda_0>0$ such that for $r\ge 0$ 
$$ \P\{l(i,j)(\o)>r\}=e^{-\lambda_0 r}.$$

Let $\Sigma$ be the set of the obstacles, i.e. 
$\Sigma := \bigcup_{(i,j) \in \Z\times \Z^*} \Big(Q_\delta(b_{i,j})\Big),$ 
then the field $f$ is defined the following way:
$$
f(x,s)=g(x,s)-\sum_{(i,j)\in \Z\times\Z^*} l(i,j)\phi((x,s)-b_{i,j})  
$$
where $g$ is a non-negative function chosen so that the field has mean zero in a suitable sense:
\begin{align*}
%&g=0 \quad \text{in} \quad \bar \Sigma\\
&g\ge 0 \quad \text{in}\quad \R^{2}\\
&\lim_{L\to\infty} (2L)^2\int_{[-L,L]^2}f(x,s)\,dxds =0
 \end{align*}
\end{definition}
\begin{remark}
\begin{enumerate}
\item
As $\E (l(i,j))=\frac{1}{\lambda},$ the law of large numbers implies that a possible choice of 
$g$ is $$g(x,s)=\sum_{(i,j)\in \Z\times\Z^*} \frac{1}{\lambda}\phi((x,s)-b_{i,j}) .$$
\item The results on non-existence of nonnegative stationary solutions hold for any
i.i.d. random variables $l(i,j)$ such that there exists $\lambda_0>0$ with
$$
 \P\{l(i,j)(\o)>r\}\le e^{-\lambda_0 r}.$$
 \item As we are only interested in the combined effect of  $f(x,s)$ and the constant forcing $F,$
 the mean zero property of the random nonlinearity is just a normalisation.
\item  In our analysis, the 
shape of the obstacles  $(supp (\phi))$ 
plays no role  and the results will stand 
as well if we consider a random field like e.g. 
$$f=g(x,s)-\sum_{(i,j)\in \Z\times\Z^*} l(i,j)\phi_{i,j}((x,s))  $$ 
where $\phi_{i,j}$ are smooth functions uniformly bounded  and 
such that    $supp(\phi_{i,j})\subset Q_{\delta}(b_{i,j})$ .  
\end{enumerate}
\end{remark}

\subsection{Results:}~
 We consider the stationary version of \eqref{cdl.eq.interface} with Dirichlet   boundary conditions:
 \begin{align}
& u_{xx} + f(x,u,\o) +F=0 \quad \text{ in } \quad [-N+\delta ,N-\delta] \label{cdl.eq.stat}\\
&u(-N+\delta)=u(N-\delta)=0 \label{cdl.eq.dirichlet}
\end{align} 
  
\begin{theo}\label{mainthm}
Let  $u(\o)$ solve (\ref{cdl.eq.stat}, \ref{cdl.eq.dirichlet}). Then there exist $F_0>0,$ $C$ and $K$ such that 
for $F>F_0$ and $N$ sufficiently large
$$
\P\left( \{\o|\, u(x,\o)\ge (K(N-1)-K|x|)_+\ {\rm on\ }[-N+\delta ,N-\delta] \}\right)\ge 1- C e^{-\frac{N}{C}}, 
$$where  $a_+$ denotes the positive part of a real number $a.$
\end{theo}
\begin{cor}Let $F>F_0,$ with $F_0$ as in Theorem \ref{mainthm}.
{\hfill }
\begin{enumerate}
\item 
There is almost surely no global nonnegative stationary solution of
(\ref{cdl.eq.interface}).
\item Let $u$ solve (\ref{cdl.eq.interface}),
(\ref{cdl.eq.initial}). Then
$$\lim_{t\to\infty}u(t,x,\omega)=+\infty\quad  {\rm for\ all\ }x\in \R$$
holds with probability one.
\end{enumerate}
\end{cor}

\section{Blocked path and auxiliary problem}
In this section we define a auxiliary problem that we will constantly use along this paper.
We will denote by  $\chi_B$ the characteristic function of the set $B.$
\begin{definition}[Auxiliary field]\label{auxfield}
%First let us choose a cut-off function $\zeta_{\delta}(s)$ with the following property:
%\begin{itemize}
%\item $\zeta_{\delta}(s)=0$ for $|s|>\delta,$ and $\zeta_{\delta}(s)=1$ for $|s|<\frac{9}{10}\delta.$
%\item  $|\zeta_{\delta}(s)|\le 1$ for all $s,$  and there exists a constant $c$ such that
%the derivatives we have \newline
%$\|\zeta_{\delta}(s)\|_{C^1}\le \frac{c}{\delta},\  \|\zeta_{\delta}(s)\|_{C^2}\le \frac{c}{\delta^2}.$
%\end{itemize} 
Let
\begin{align*}
&A:= \R^2\setminus \{\bigcup_{i\in \Z}(i-\delta,i+\delta)\times\R \}\\
&A_\eps:= \R^2\setminus \{\bigcup_{i\in \Z}(i-\delta-\eps,i+\delta+\eps)\times\R \}
\end{align*}

and define
$$\tilde f(x,s):=-\sum_{(i,j)\in \Z^*\times\Z^*} l(i,j)\phi((x,s)-b_{i,j}).$$
\end{definition}
Let us now consider the following auxiliary problem 
\begin{align}
 &\frac{\partial v}{\partial t}\,=\,v_{xx} + \tilde f(x,v(t,x)) +F\chi^\eps_A(x) \label{cdl.eq.approx}\\
 &v(0,x) \,=\,0,
\end{align}
where $\chi_A^{\eps}$ is a smooth function such that $ \chi_{A_\eps}\le\chi_A^{\eps}\le \chi_A$. $\eps$ is a small parameter which will be fix later on.

To visualize the new random field defined by $\tilde g= \tilde f+F\chi^\eps_A(x)$ see figure 2.  
%In particular, we have $\tilde f$ on the grey part of the figure.  
Note that it is differentiable in $x$ and $s.$
\begin{figure}
\begin{center}
\input{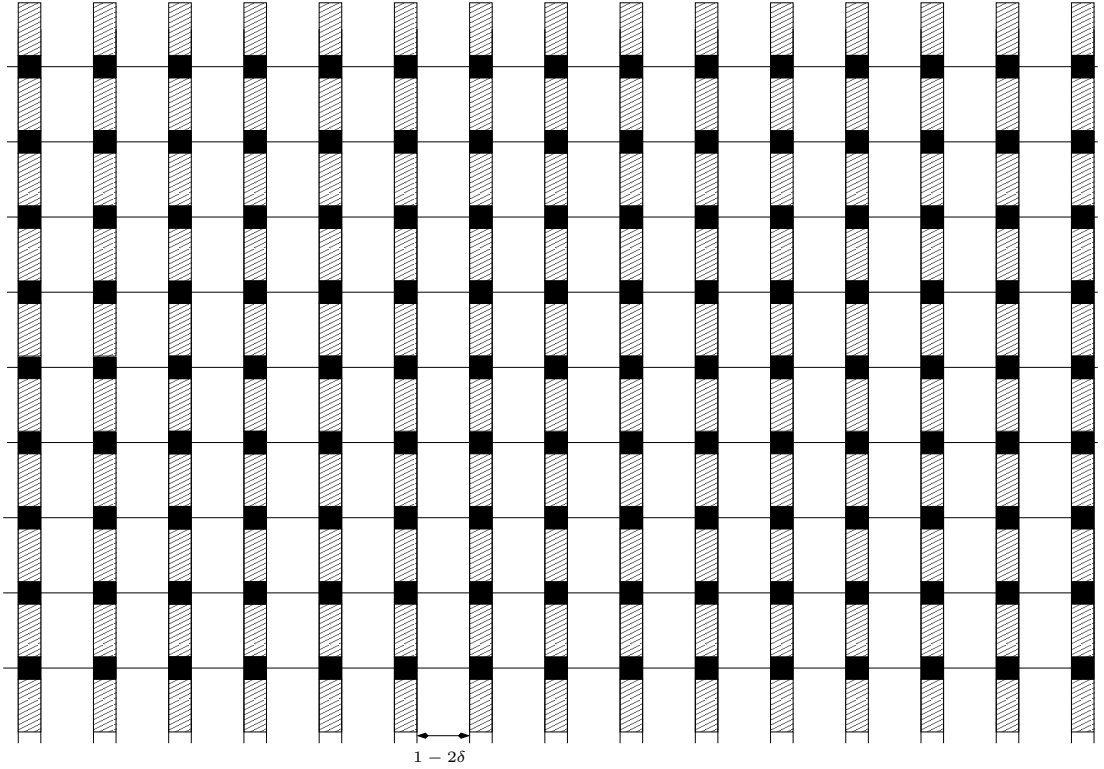}
\caption{Mapping of the obstacles for the auxiliary problem}
\end{center}
\end{figure}
%\textcolor{red}{Definir le champs et pipoter sur les blocking path, faire un dessin}

Observe that, as the obstacles are negative,  $\tilde f+F\chi^\eps_A  \le f+F.$ Therefore the comparison principle
for the parabolic equation (see section \ref{existence}) implies that solutions of the auxiliary problem
remain below solutions of the original problem. Hence existence of a nonnegative stationary solution for
the original problem implies existence of one for the auxiliary problem. By contraposition, nonexistence
for the auxiliary problem implies nonexistence for the original problem.
%if for some $F$ the interface $\G$ defined by \eqref{cdl.interface} is pinned then the corresponding interfaces $\tilde \G%$ defined by 
%\begin{equation}\label{cdl.auxillary-interface}
%\tilde\Gamma(t):=\{ (x,y)\in \R^{2}: y=v(x,t)\},
%\end{equation}
%where $v$ is solution of \eqref{cdl.eq.approx} is also pinned. 

Stationary sub/supersolutions can be constructed as piecewise quadratic functions.
For any $F$ we can construct the graph of such a solution 
(also called "paths" to emphasize the analogy with a stochastic process).
\begin{definition}[blocked path] \label{blocked} A
 graph $(x,v(x))$ is called blocked path if and only if
$v \in  C^1_{loc}(\R),$ and 
\begin{align}
&v_{xx}=-F\chi_A^\eps(x) \qquad \text{ in } \qquad (i+\delta ,i+1-\delta),\label{cdl.eq.outside.strip}\\ 
&v_{xx}= \sum_{j\in \Z^*} l(i,j)(\o)\phi_{i,j}(x,v(x))  \qquad \text{ in } \qquad (i-\delta , i+\delta)\label{cdl.eq.inside.strip}. 
\end{align}
where $\phi_{i,j}(x,s):=\phi((x,s)-b_{i,j})$.
\end{definition}
Observe that the path for $x\in (i+\delta ,i+1-\delta)$ is uniquely determined by the  boundary values 
 $v(i+\delta)$ and $v(i+1-\delta),$ because it solves a {\em linear} elliptic equation there.
But note that, for a given realisation of the random field, there may be more
than one blocked path, as equations like $u_{xx}=f(x,u)$ do not have unique solutions without further conditions on
the nonlinearity.
\begin{remark}
>From Definition \ref{blocked}, we see that $v$ is a convex function in $(i-\delta , i+\delta)$ and hence we have
$$v(i+\delta)\ge v(i-\delta)+2\delta v'(i-\delta) $$ 
\end{remark}

Let us now define some discrete quantities that we will use throughout the paper.

\begin{definition}\label{hat_v}
Let $\hat v(i)$ and $\bar v^{\delta}[i]$ be defined as follows:
 $$\hat v(i):=v(i-\delta)+2\delta v_{x}(i-\delta),$$ 
$$\bar v^{\delta}[i]:=\delta \left[\delta^{-1}\hat v(i)-\frac{1}{2} \right]=\inf\{j\in \delta\Z\, |\,j\ge \hat v(i)-\frac{\delta}{2} \}
 \in\delta\Z.$$
\end{definition}

We will need the following Lemma.

\begin{lem}\label{comparisonlemma}
Let $v$ be as in Definition \ref{blocked} and  
$\hat v,\ \bar v^\delta$ be  in Definition \ref{hat_v}. Denote by $\bar w^\delta$ the
piecewise linear interpolation of $\bar  v^\delta,$ and by $w$ the piecewise linear interpolation
of $\hat v.$ Then
$v+\delta/2\ge \bar w^\delta,$ and $v\ge w.$
 \end{lem}
\begin{proof}
First, note that convexity of $v$ in $[i-\delta,i+\delta]$ implies
that $\hat v(i)\le v(i+\delta).$ 

%We use the comparison
%principle for the Laplace operator on each  of the intervals 
%$$I_i:=(i-1+\delta,i+\delta)$$ separately. 
Let $$I_i:=(i-1+\delta,i+\delta)$$ and let the auxiliary function 
$\hat w$ be the solution of
\begin{eqnarray*}
&&\Delta \hat w=-F1_{[i-1+\delta,i-\delta]}\quad {\rm on\ }I_i\\
&&\hat w(i-1+\delta)=v(i-1+\delta),\quad \hat w(i+\delta)=\hat v(i).
\end{eqnarray*}
This function is $C^{1}$ on its domain and solves the ODE
$\hat w_{xx}=-F$ on $(i-1+\delta, i-\delta).$ (Here $x$ is considered as "time"). Suppose 
$\hat w(i-\delta)>v(i-\delta).$ Then $\hat w_{x}(i-\delta)<
v_{x}(i-\delta),$ and integrating the ODE backwards in $x$ we obtain
$\hat w(i-1+\delta)>v(i-1+\delta),$ a contradiction. Assuming  
$\hat w(i-\delta)<v(i-\delta)$ we obtain a contradiction in a similar
way,
and we conclude $\hat w(i-\delta)=v(i-\delta).$ This implies 
%$\hat w_{x}(i-\delta)=
%v_{x}(i-\delta),$ and therefore $\hat v(i):=\hat w(i+\delta).$
  $\hat w=v$ on $[i-1+\delta,i-\delta]$ and (by convexity
of $v$ on $[i-\delta,i+\delta]$) $\hat w\le v$ on
$[i-1+\delta,i-\delta].$

Now consider
\begin{eqnarray*}
&&\Delta w=0\quad {\rm on\ }I_i\\
&&w(i-1+\delta)=\hat v(i-1),\quad  w(i+\delta)=\hat v(i)
\end{eqnarray*}
Clearly $w$ is the piecewise linear interpolation of 
$\hat v.$

As $\Delta \hat w \le \Delta w$ and $w\le \hat w$ on $\partial I_i,$ 
the comparison principle for the Laplace 
operator
implies $\hat w\ge w,$ so
$v\ge \hat w\ge w.$ The conclusion  for $\bar w^\delta$ follows immediately.

\end{proof}

\subsection{Existence and uniqueness for parabolic equations}\label{existence}
\begin{lem}
There exists a global classical solution of the parabolic Cauchy problems (\ref{cdl.eq.interface}),
and   (\ref{cdl.eq.approx}) with initial conditions which are uniformly bounded and locally
$C^2.$
The solutions are unique. If $0\le v_0\le u_0,$ $v$ solves  (\ref{cdl.eq.approx}) with initial condition
$v_0,$ $u$ solves (\ref{cdl.eq.interface}) with initial condition $u_0,$ then  $v \le u.$
\end{lem}
Proof: For $M\in \N, $  replace $l(i,j)(\omega)$ by $l^M(i,j):=M\wedge l(i,j)$, where $\wedge$ denotes the operation $ a\wedge b:=\inf\{a,b\} $.The corresponding fields
$f^M,\ \tilde f^M$ are uniformly bounded and uniformly Lipschitz in $s.$ Therefore we can apply the Banach
fixed point theorem in $L^\infty$ in order to obtain a local in time solution, which,
by local parabolic regularity, is classical. It can be extended as the nonlinearity is uniformly
bounded. Hence a global solution $u^M(x,t)$ exists. Note that by the comparison principle $u^{M}$ is a positive monotonic non-increasing function of $M$ i.e. $u^M>u^N>0$ for $N>M$ , so $u(x,t):=\lim_{M\to\infty}u^M(x,t)$ exists.
Applying regularity locally, (where the obstacles are bounded) we obtain that the limit is a classical solution.
%%%%%%%%%%%%%%%%%%%%%%%%%%%%%%%%%%%%%%%%%%%%%%%%%%%%%%%
\section{\textit{a priori} estimates on $\hat v(i)$ and $\bar v^{\delta}[i]$}
In this section, we establish some a-priori estimates on $\hat v(i)$  and $\bar v^{\delta}[i]$. 

First we  show a lemma which allows to estimate the discrete Laplacian of $\hat v$ at $i$ (which involves
$i,i+1$ and $i-1$) by something that depends only on the obstacles above $i.$ 
\begin{lem}\label{cdl.lem.esti1}
Let $\hat v(i)$ defined as in the previous section, and define the discrete Laplacian as
$$\Delta_{d}\hat v(i):= \hat v(i+1)-2\hat v(i)+\hat v(i-1) =\big(\hat v(i+1)-\hat v(i)\big)-\big(\hat v(i) -\hat v(i-1)\big)$$
Then 
%$$\Delta_{d}\hat v(i)\le (1+2\delta)[v_{x}(i+\delta)-v_{x}(i-\delta)]-F(1-2\delta).$$
$$
-2\delta[v_{x}(i-1+\delta)-v_{x}(i-1-\delta)]\le \Delta_{d}\hat v(i)+
\hat F\le (1+2\delta)[v_{x}(i+\delta)-v_{x}(i-\delta)].$$
where $F(1-2(\delta+\eps))\le\hat F \le (1-2\delta)F$ for the  $\eps>0$ in Def. \ref{auxfield}.
\end{lem}
Note that our discretization, using the tangents, implies that the discrete Laplacian does not
necessarily satisfy the same lower bound as the Laplacian of the original path.
\dem{Proof:} {\bf Step One : Upper Bound}\\
 As a preparation, let us recall some formulas satisfied by $v$.\\
Since $v$ satisfies \eqref{cdl.eq.outside.strip}, we have for all $i\in \Z$
\begin{align}
&v_{x}(i+1-\delta)-v_{x}(i+\delta)=-F\int_{i+\delta}^{i+1-\delta}\chi_A^{\eps}(x)dx \label{cdl.eq.esti1}\\
&v(i+1-\delta)-v(i+\delta)=(1-2\delta)v_{x}(i+\delta)-F\int_{i+\delta}^{i+1-\delta}\left(\int_{i+\delta}^{s}\chi_A^{\eps}(x)dx\right)\,ds \label{cdl.eq.esti2}
\end{align}
Let us define $$\hat F:=F\int_{i+\delta}^{i+1-\delta}\chi_A^{\eps}(x)dx. $$
Observe that since  $\chi_A^{\eps}(x+p)=\chi_A^{\eps}(x)$ for all integer $p$, $\hat F$ is independant of $i\in \Z$. Moreover
$$F(1-2(\delta+\eps))\le \hat F\le (1-2\delta)F$$ since $\chi_{A_{\eps}}(x)\le \chi_A^{\eps}(x)\le\chi_A(x) $.

Using now \eqref{cdl.eq.esti1}, the definition of $\hat v(i+1)$ and \eqref{cdl.eq.esti2} we see that  
\begin{align*}
 \hat v(i+1)&=v(i+\delta)+v_{x}(i+\delta)-F\int_{i+\delta}^{i+1-\delta}\left(\int_{i+\delta}^{s}\chi_A^{\eps}(x)dx\right)\,ds+2\delta (v_{x}(i+1_\delta)-v_{x}(i+\delta))\\
 &=v(i+\delta)+v_{x}(i+\delta)-F\int_{i+\delta}^{i+1-\delta}\left(\int_{i+\delta}^{s}\chi_A^{\eps}(x)dx\right)\,ds-2\delta \hat F.
 \end{align*}
 Therefore, 
 \begin{equation}\label{cdl.eq.esti3}
 \hat v(i+1)-\hat v(i)= v(i+\delta)+v_{x}(i+\delta)-F\int_{i+\delta}^{i+1-\delta}\left(\int_{i+\delta}^{s}\chi_A^{\eps}(x)dx\right)\,ds-2\delta \hat F  -\hat v(i).
 \end{equation} 
 Observe that since  $\chi_A^{\eps}(x+p)=\chi_A^{\eps}(x)$ for all integer $p$ we have 
  $$F\int_{i+\delta}^{i+1-\delta}\left(\int_{i+\delta}^{s}\chi_A^{\eps}(x)dx\right)\,ds+2\delta \hat F =F \int_{i-1+\delta}^{i-\delta}\left(\int_{i-1+\delta}^{s}\chi_A^{\eps}(x)dx\right)\,ds+2\delta \hat F.$$
 Hence, from the definition of the discrete laplacian and using \eqref{cdl.eq.esti3} it follows that  
 \begin{equation}
 \Delta_d \hat v(i)= v(i+\delta)+v_{x}(i+\delta)-\hat v(i) - v(i-1+\delta)-v_{x}(i-1+\delta)+\hat v(i-1)\label{cdl.eq.esti4}
 %&= (v(i+\delta)-v(i-\delta))+v_{x}(i+\delta)-2\delta v_{x}(i-\delta)-v^{\prime}(i-1+\delta) - (v(i-1+\delta)-v(i-1-\delta)+2\delta v^{\prime}(i-1-\delta).
 \end{equation}
 
 Using now the definition of $\hat v(i)$ and the convexity of $v$  in $(i-\delta,i+\delta)$ for all $i\in\Z$ we see that 
    \begin{align*}
 &v(i+\delta)+v_{x}(i+\delta)-\hat v(i) \le  v_{x}(i+\delta)+ 2\delta (v_{x}(i+\delta)-v_{x}(i-\delta))\\
 &- v(i-1+\delta)+\hat v(i-1) \le 0.
 \end{align*}
 Hence, 
 $$  
 \Delta_d \hat v(i)\le  (1+2\delta)v_{x}(i+\delta)-2\delta v_{x}(i-\delta)-v_{x}(i-1+\delta).$$
Using now \eqref{cdl.eq.esti1} it follows that 
$$  
 \Delta_d \hat v(i)\le  (1+2\delta)(v_{x}(i+\delta)- v_{x}(i-\delta)) -\hat F.$$

{\bf Step two: Lower bound}

 From  formula \eqref{cdl.eq.esti4}  we have 
\begin{equation}
\Delta_d \hat v(i)=v(i+\delta) -\hat v(i) +v_{x}(i+\delta)-v(i-1+\delta) +\hat v(i-1) -v_{x}(i-1+\delta)
\end{equation}

Since $v$ is convex in $(i-\delta,i+\delta)$, we have  $v(i+\delta) -\hat v(i) \ge 0$ and $v_{x}(i+\delta)\ge v_{x}(i-\delta)$. Therefore we have 
\begin{equation}
\Delta_d \hat v(i)\ge  v_{x}(i-\delta)-v_{x}(i-1+\delta)-v(i-1+\delta) +\hat v(i-1). 
\end{equation}
Using now \eqref{cdl.eq.esti1}, the convexity of $\hat v$ in $(i-1-\delta,i-1+\delta)$ and the definition of $\hat v(i-1)$ it follows that 
 \begin{equation}
\Delta_d \hat v(i)\ge -\hat F- 2\delta [v_{x}(i-1+\delta) -v_{x}(i-1-\delta)]. 
\end{equation}
\fdem

Now we proceed to estimate  the change of the discrete gradient 
$$k(i):=v_{x}(i+\delta)-v_{x}(i-\delta)$$ in terms of the obstacle strengths above $i.$
Observe that  always $k\ge 0$ by convexity.
If the gradients are very steep, the path will pass through several obstacles above the interval
$[i-\delta,i+\delta].$ The number of obstacles passed and the time spent in each
of them (i.e. the Lebesgue measure of its image under the inverse mapping) 
can be estimated in terms of $v'(i-\delta)$ and $v'(i+\delta).$  

%\begin{lem}\label{cdl.lem.esti2}
%Let $v$ be a blocked path, $i\in \Z$ and assume that $k(i)>0$. Then the following holds
%\begin{enumerate}
%\item \underline{Case one:}\quad $v^{\prime}(i-\delta)<0$; $|v^{\prime}(i-\delta)|\ge |v^{\prime}(i+\delta)|$.\\
%Then we have $$k\le \frac{18\delta}{|v^{\prime}(i-\delta)|}\sum_{\hat v(i)\le j\le \hat v(i)+
%4\delta|v^{\prime}(i-\delta)|}l(i,j).$$
%\item \underline{Case two:}\quad $v^{\prime}(i-\delta)<0$; $|v^{\prime}(i-\delta)|< |v^{\prime}(i+\delta)|$.\\
%Then we have $$k\le \frac{18\delta}{|v^{\prime}(i+\delta)|}\sum_{\hat v(i)\le j\le \hat v(i)
%+4\delta|v^{\prime}(i+\delta)|}l(i,j).$$
%\item \underline{Case three:}\quad $ v^{\prime}(i-\delta)\ge 0$.\\
%Then we have 
%\begin{align*}
%k&\le \frac{18\delta}{|v^{\prime}(i+\delta)|}
%\sum_{\hat v(i)-2\delta v^{\prime}(i-\delta)\le j\le \hat v(i)+4\delta|v^{\prime}(i+\delta)|}l(i,j)\\
%&\le \frac{18\delta}{|v^{\prime}(i+\delta)|}
%\sum_{\hat v(i)-4\delta |v^{\prime}(i+\delta)|\le j\le \hat v(i)+4\delta|v^{\prime}(i+\delta)|}l(i,j)
%\end{align*}

%\end{enumerate}

%\end{lem}
%From the above Lemma, we can easily deduce the following estimate
\begin{lem}\label{cdl.lem.esti3}
Let $v$ be a blocked path, $i\in\Z$ and and assume that $k(i)>0$.
Set $M:=sup\{|v_{x}(i-\delta)|;|v_{x}(i+\delta)|\}$ then we have
$$k(i)\le \frac{18\delta}{M}\sum_{\hat v(i)-4\delta M\le j\le \hat v(i)+4\delta M}l(i,j)$$
\end{lem} 
\dem{Proof:}
%In the first case we have $M=|v^{\prime}(i-\delta)|$ and from Lemma \ref{cdl.lem.esti2} we have
%$$ k\le \frac{18\delta}{|v^{\prime}(i-\delta)|}\sum_{\hat v(i)\le j\le \hat v(i)+4\delta|v^{\prime}(i-\delta)|}l(i,j). $$
%Therefore, since $M>0$ and $l(i,j)>0$ we have
%\begin{align*}
% k&\le \frac{18\delta}{M}\sum_{\hat v(i)\le j\le \hat v(i)+4\delta M}l(i,j)\\
%&\le \frac{18\delta}{M}\sum_{\hat v(i)-4\delta M\le j\le \hat v(i)+4\delta M}l(i,j)
%\end{align*}
%The proof in the other case is similar and is left to the reader.

Step 1:
As $v$ is convex on $[i-\delta,i+\delta],$ the gradient is monotone, hence $|v_x(x)|\le M$
for all $x\in I(i):=[i-\delta,i+\delta].$ As a consequence, we have on $I(i)$
$$
v(i)-\delta M\le v(x)\le v(i)+ \delta M,
$$ As $|\hat v(i)-v(i-\delta)| \le 2\delta M,$ $|v(i)-v(i-\delta)|\le \delta M,$ we obtain 
$$
|v(x)-\hat v(i)|\le 4\delta M\quad {\rm on }  [i-\delta,i+\delta].
$$

Step 2.
Define the time spent by the path in the $j$-th obstacle above $i$ as 
$$S_j:=\big|\{x:\ v(x)\in [j-\delta,j+\delta] \}\big|,
$$
where $|A|$ denotes the Lebesgue measure of the set $A$ and $j\in Z^*=1/2+\Z.$ Note that by convexity
$v_x$ changes sign at most once, hence each $S_j$ is the union of at most two intervals, moreover
$S_j=\emptyset$ if $|j-\hat v(i)|>4\delta M$
%5 \delta M????
Hence, as for $x\in I(i)$ $v_{xx}(x) \le l(i,j)$ on obstacle $j$ and zero else, 
$$
v_x(i+\delta)-v_x(i-\delta)\le \sum_{\hat v(i)-4\delta M\le j\le \hat v(i)+4\delta M}l(i,j)S_j.
$$where $j\in \Z^*$
%It remains to bound the $\tau_j$ from below.

Step 3. Note that $k\le 2M.$ 
As the gradient is monotone on $I(i),$ there exists a $\hat \tau$ such that 
$|v_x(\hat \tau)|=M-k/3$ and $|v_x(x)|\ge M-k/3\ge M/3\ge 0$ on $\hat I(i),$ where

$$
\hat I(i)=\left\{
\begin{array}{ll} [\hat \tau,i+\delta]  & {\rm if}\ M=|v_x(i+\delta)|  \\ 
{} [i-\delta,\hat \tau]  & {\rm if}\ M=|v_x(i-\delta)|. \end{array}
\right.
$$
As the gradient does not change sign on $\hat I (i),$ the sets $\hat S_j:=S_j\cap \hat I(i)$ are intervals.
Moreover, 
$$|\hat S_j|\le \frac{2\delta}{M/3}=\frac{6\delta}{M}
$$ as $|v_x|\ge M/3$ on $\hat I(i).$
Hence
$$
\frac{k}{3}=M-v_x(\hat \tau)\le \sum_{\hat v(i)-4\delta M\le j\le \hat v(i)+4\delta M}l(i,j)\hat S_j\le \frac{6\delta}{M}
\sum_{\hat v(i)-4\delta M\le j\le \hat v(i)+4\delta M}l(i,j)
$$
and the result follows.
\fdem
\begin{remark}\label{cdl.rem.m(i)}
Note that in the case where $k(i)\ge 1$  then the corresponding $M(i)\ge \frac{1}{2}$. Indeed, by definition of $M(i)$ and $k(i)$ we have $2M(i)\ge |v_{x}(i-\delta)|+|v_{x}(i+\delta)|\ge v_{x}(i+\delta) -v_{x}(i-\delta)=k(i)\ge 1 ,$
i.e. $M(i)\ge 1/2.$ 
\end{remark}

Combining now Lemmas \ref{cdl.lem.esti1} and \ref{cdl.lem.esti3}  we deduce the following estimates, which
allow to estimate the discrete Laplacian of the blocking path $(\bar v^{\delta}[j])_{j\in [-N,N]\cap\Z}$ at a site $i$ against a normalized sum of random variables.

\begin{lem}\label{cdl.lem.esti5}
Let $v$ be a blocked path. 
Then for all $i \in [-N+\delta,N-\delta]\cap \Z$ there exists  $M(i), M(i-1)>\frac{1}{2}$ such that following  holds
  \begin{align*}
   \Delta_d \bar v(i) +\bar F  &\le(1+2\delta)\left[\frac{360\delta^2}{2\delta(4 M(i)+\frac{1}{2})}\sum_{\bar v(i)-\delta(4 M(i)+\frac{1}{2})\le j\le \bar v(i) +\delta(4 M(i)+\frac{1}{2})}l(i,j)(\o)\right]\\
   &\ge -2\delta\left[\frac{360\delta^2}{2\delta(4 M(i-1)+\frac{1}{2})}\sum_{\bar v(i-1)-\delta(4 M(i-1)+\frac{1}{2})\le j\le \bar v(i-1) +\delta(4 M(i-1)+\frac{1}{2})}l(i-1,j)(\o)\right] ,
  \end{align*}
where $\bar F:=\hat F -(1+2\delta)$.
\end{lem}
% \begin{lem}\label{cdl.lem.esti5}
%Let $v$ be a blocked path. Then the   following estimates holds
%\begin{itemize}
%\item Either $k\le 1$ and $$-4\delta\le \Delta_d \bar v^{\delta}[i]+F(1-2\delta)\le 1+4\delta .$$
%\item Or $k> 1 $ and there exists $M>\frac{1}{2}$ such that
%  \begin{eqnarray*}
%  &&\Delta_d \bar v^{\delta}[i]+F(1-2\delta)
%  \le (1+2\delta)\left[\frac{180\delta^2}{2(4M+\frac{1}{2})\delta}\sum_{\bar v^{\delta}[i]-(4M+\frac{1}{2})\delta\le j
%  \le \bar v^{\delta}[i] +(4M+\frac{1}{2})\delta}l(i,j)(\o)\right] +2\delta.\\
%   &&\Delta_d \bar v^{\delta}[i]+F(1-2\delta)\ge 
%- \left[\frac{180\delta^2}{(4M+\frac{1}{2})}\sum_{\bar v^{\delta}[i]-(4M+\frac{1}{2})\delta\le j
%  \le \bar v^{\delta}[i] +(4M+\frac{1}{2})\delta}l(i-1,j)(\o)\right] -2\delta.
%  \end{eqnarray*}
%\end{itemize}
%\end{lem}

\dem{Proof:} Let us first start with the proof of the upper bound.
Observe first that  
$$ \hat v(i)-\frac{\delta}{2}\le \bar v^{\delta}[i]\le \hat v(i)+\frac{\delta}{2},$$
which implies that  
$$ \Delta_d \hat v^{\delta}[i] -2\delta \le \Delta_d \bar v^{\delta}[i]\le \Delta_d \hat v^{\delta}[i] +2\delta.$$
Therefore using Lemma \ref{cdl.lem.esti1} we have 
\begin{equation}\label{cdl.eq.esti-laplaced3}
\Delta_d \bar v^{\delta}[i]\le (1+2\delta)k(i) -\hat F +2\delta.
\end{equation}
with $k(i)>0$.
  By Lemma \ref{cdl.lem.esti3} and Remark \ref{cdl.rem.m(i)}, for $k(i)\ge 1$ there exists $M(i)\ge \frac{1}{2}$ so that
$$k(i)\le\frac{18\delta}{M(i)}\sum_{\hat v(i)-4\delta M(i)\le j\le \hat v(i) +4\delta M(i)}l(i,j)(\o).$$
So we easily see that
\begin{equation}\label{cdl.eq.esti-laplaced4}
k(i)\le\frac{18\delta^2(4M(i)+\frac{1}{2})}{M(i)(4M(i)+\frac{1}{2})\delta}\sum_{\bar v^{\delta}[i]-(4 M(i)+\frac{1}{2})\delta\le j\le \bar v^{\delta}[i] +(4 M(i)+\frac{1}{2})\delta}l(i,j)(\o).
\end{equation}
Therefore,  since $M(i)>\frac{1}{2}$ we have
\begin{equation}\label{cdl.eq.esti-laplaced5}
k(i)\le\frac{180\delta^2}{(4M(i)+\frac{1}{2})\delta}\sum_{\bar v^{\delta}[i]-(4 M(i)+\frac{1}{2})\delta\le j\le \bar v^{\delta}[i] +(4 M(i)+\frac{1}{2})\delta}l(i,j)(\o).
\end{equation}

Hence, for all $k(i)\ge 0$, we have 
$$k(i)\le 1+ \frac{180\delta^2}{(4M(i)+\frac{1}{2})\delta}\sum_{\bar v^{\delta}[i]-(4 M(i)+\frac{1}{2})\delta\le j\le \bar v^{\delta}[i] +(4 M(i)+\frac{1}{2})\delta}l(i,j)(\o).$$
and the estimate follows .
The lower bound is treated in a similar way.
%since $M>\frac{1}{2}$
\fdem
%\begin{remark}\label{cdl.rem.esti}
%Observe that from the  proof  the estimate in Lemma \ref{cdl.lem.esti5} can be condensed into a unique formula. %Indeed, since the $l(i,j)$ are positive random variables and $1+2\delta>0$, we  have for some $M\ge \frac{1}{2}$ such %that $(4M+\frac{1}{2})\delta\ge 1$
%$$\big| \Delta_d \bar v^{\delta}[i] +F(1-2\delta)\big|\le \left[\frac{C_0}{(4M+\frac{1}{2})\delta}
%\hspace{-0.3cm}\sum_{\bar v^{\delta}[i]-(4M+\frac{1}{2})\delta\le j\le \bar v^{\delta}[i] +(4M+\frac{1}{2})\delta}
%\hspace{-0.3cm}l(i,j)(\o)\right]+C_1 l(i,j_0[i])(\o) +2\delta +(1+2\delta)$$
%where 
%\begin{align*}
%$C_0:=180\delta^2(1+2\delta)$  %\\
%$C_1:=(1+2\delta)135\delta.$
%\end{align*}  
%\end{remark}  
%\marginpar{ Why the $l(i,j_0)$?}
%%%%%%%%%%%%%%%%%%%%%%%%%%%%%%%%%%%%%%%
\section{Probabilistic Estimates}

We first recall a standard fact for the Laplace transform of independent exponential random variables and random variables with distribution function bounded by an exponential.
\begin{lem}\label{cdl.lem.prob1}
\begin{enumerate}
\item
Let $\{X_i\}_{i\in N}$ be independent identically distributed 
random variables such that for a parameter $\lambda_0$
and a constant $C>0$
\begin{equation}\label{cdl.eq.expbound}
{\mathbb P}[X_0>r]\le Ce^{-\lambda_0 r}.
\end{equation}
Then we have for any $\lambda<\lambda_0$ and $L\in \N,$ $L\ge 2,$
\begin{eqnarray}\label{cdl.eq.L=1}
\E\left[ e^{\lambda X_1}\right]&\le&C\frac{\lambda_0}{\lambda_0-\lambda}\\ \label{cdl.eq.Lnot1}
\E\left[ e^{\lambda \sum_{i=1}^L X_i}\right]&\le&C^L\left(\frac{\lambda_0}{\lambda_0-\lambda}\right)^L\\
\label{cdl.eq.Lge1}
\E\left[ e^{\lambda \left(\frac{1}{L}\sum_{i=1}^L X_i\right)}\right]&\le&C^Le^{\lambda \frac{4\ln(4/3)\lambda}{3\lambda_0}}
\quad {\rm \ for\ } L\ge 2 ,\lambda\in (2/3\lambda_0,\lambda_0) 
\end{eqnarray}
\item Let $\{X_i\}_{i\in N}$ be independent exponential random variables with parameter $\lambda_0>0.$
Then (\ref{cdl.eq.L=1})-(\ref{cdl.eq.Lnot1}) hold as equalities with $C=1,$ while 
 (\ref{cdl.eq.Lge1}) holds as inequality with $C=1.$
\end{enumerate}
\end{lem} 
\dem{Proof:} We first show 2.
The first equality is standard, the second follows by using independence. For the third,
note that by concavity of $\ln(1-x)$ on $[0,3/4]$ 
$$\ln(1-x)\ge \frac{4}{3}x \ln(3/4)\ {\rm for}\ x\in \left[0,\frac{3}{4}\right]$$ Using independence and this concavity estimate 
with $x=\lambda_0/(\lambda L)$
$$
\E\left[ e^{\lambda \frac{1}{L}\sum_{i=1}^L X_i}\right]=\left(\frac{\lambda_0}{\lambda_0-\frac{\lambda}{L}}\right)^L
=e^{-L\ln\left(1-\frac{\lambda}{\lambda_0L}\right)}\le e^{\ln(4/3)\frac{4\lambda}{3\lambda_0}}.
$$

In order to show 2., it is sufficient to prove the first inequality, the others then follow as in the previous case.
For (\ref{cdl.eq.L=1}) note that the expectation of a random variable is the Riemann-Stieltjes integral with the
distribution function as integrator. Now integrate by parts and use that the integrand $e^{\lambda x}$ 
is monotone.

\fdem

\begin{remark} \label{cdl.rem.esti.expectation}
Observe that the above estimate on the Laplace transform of $S_L$ is independent of $L$. 
\end{remark}

Let us define $\tilde S_M$ by $$\tilde S_M(\o)(i,j):= \sum_{-M\le j-l\le M}l(i,l).$$
The  we have the following Corollary:

\begin{cor}\label{corprob}
For any discrete function $j(i):\Z \to\Z,$ the random variables $\{\tilde S_M(\o)(i,j(i))\}_{i\in \Z}$ 
are independent and identically distributed. Moreover,
there exist constants $C,\hat \lambda$ which depend only on $\lambda_0$ such that
$$
\P\left( \tilde S_M(\o)(i,j(i))>r\right)\le e^{C-\hat\lambda r}
$$
\end{cor}
\dem{Proof:} The first assertion is obvious. The second
is a consequence of (\ref{cdl.eq.L=1}) and (\ref{cdl.eq.Lge1}) and the exponential Chebyshev inequality
with a parameter $\lambda\in (2/3\lambda_0,\lambda_0).$
\fdem

%\marginpar{I made the modification we discussed}

Let us now estimate  the probability of a blocked path with boundary
conditions on $[-N,N]$ 
to be compatible with the $l(i,j)$. 
\begin{definition}\label{BlockedDirichlet} (blocked Dirichlet path)
\noindent
Let $v(-N+\delta)=v(N-\delta)=0.$ 

\noindent
Moreover, let $v$ solve 
(\ref{cdl.eq.outside.strip}) for $-N\le i\le N-1,$ and let
$v$ solve (\ref{cdl.eq.inside.strip}) for $-N+1\le i \le  N-1.$ 

\noindent
Extend $v$ to $[-N-\delta,N+\delta]$ by 
$$v(x)=v'(-N+\delta)(x+N-\delta) {\quad \rm on \quad}
[-N-\delta,-N+\delta]$$  and 
 $$v(x)=v'(N-\delta)(x-N+\delta){\quad \rm on \quad}
[N-\delta,N+\delta].$$
\end{definition}
\begin{remark}\label{remdiscrete}{\quad \hfill {\quad}}
\begin{enumerate}
\item 
Note that this path solves  (\ref{cdl.eq.inside.strip}) 
for $-N\le i \le  N$ if we set $l(i,j)=0$ for $i=-N$ or $i=N.$ 
\item
If $v\ge0$ on $[-N+\delta,N-\delta],$ then $$0\ge 
v(x)\ge -2\delta FN\quad {\rm for}\quad 
x\in  [-N-\delta,-N+\delta]\cup[N-\delta,N+\delta].$$
\end{enumerate}
\end{remark}

\begin{definition}
Let $\bar v^\delta:\ [-N,N]\cap \Z\to \delta\Z$ 
be a discrete path. We call the path {\em compatible} with a random
obstacle configuration if there exists a (not necessarily unique)
path as in Definition \ref{BlockedDirichlet} which is mapped to 
$\bar v^\delta$ 
under the discretization defined in Def. \ref{hat_v}.
\end{definition}
Note that the discrete path is fixed. Whether it is compatible or not depends on the configuration
of the random field.

\begin{lem}\label{cdl.lem.prob3}
Let $(\O,\f,\p)$ be a probability space and let 
$l(i,j)(\o)$ be i.i.d. exponential  random variables 
with parameter $ \lambda_0>0$ and let $\bar v^{\delta}$ be a discrete
path
with fixed boundary conditions
$$\bar v^{\delta}(-N+\delta)=0,\ \bar v^{\delta}(N)=b\  {\rm for\ some\ }
b\in [-FN,FN].$$
Then there exist constants 
$\hat C(\delta,\lambda_0),\ \lambda_1(\delta,\lambda_0)$ independent
of
$b$  
such that we have for $F$ sufficiently large %and $\lambda\le\lambda_0/2$
$$
\P[\bar v^\delta\  {\rm compatible},\  \bar v^\delta(N)=b]
:=\P_b[\bar v^\delta\  {\rm compatible}]\le e^{N \hat C}
e^{-\lambda_1\sum_{-N+1}^{N-1} |\Delta_d\bar v^\delta(i)+\bar F|}.
$$ with $\bar F$ as in Lemma \ref{cdl.lem.esti5}
 \end{lem}
 
The previous estimates bounds the probability of the random obstacle configurations such that a {\em fixed} 
discrete path  is compatible with the random environment. In order to prove that the probability that 
{\em there exists} some  compatible nonnegative path is small, 
we would
have to sum over all possible paths, 
each weighted with the right hand side of the previous estimate.
It is complicated to bound these sums, because the number of possible discrete paths grows faster than 
exponentially in 
$N.$ Fortunately, most of them are extremely unlikely to be compatible. In order to quantify this,
we define an auxiliary probability measure on discrete paths. 

\begin{definition}\label{Ptilde}
\begin{eqnarray*}
\widetilde \P_b[\Delta_d\bar v^\delta]&:=&
\frac{1}{Z^{2N-1}}e^{-\lambda_1\sum_{-N+1}^{N-1} |\Delta_d\bar v^\delta(i)+\bar F|},
\\
Z&:=&\sum_{k=-\infty}^\infty e^{-\lambda_1|\delta k+\bar F|}
\end{eqnarray*}
\end{definition}
The normalisation constant is obtained by summing over all possible discrete paths for fixed boundary
conditions. This is equivalent to summing over all discrete Laplacians. Note that $Z$ is bounded from above
and below by constants independent of $F.$ 

Note that the law of the positive and the negative part 
of $\Delta_d\bar v^\delta(i)+\bar F$ 
 under $\widetilde P$ is that of (discretized) independent exponential 
random variables.
In particular, probabilities of sums of the discrete Laplacians have certain exponential moments and can
be estimated by large deviation techniques.

\begin{cor}\label{cdl.cor.auxmeas}
With $\widetilde P$ as in Def. \ref{Ptilde}, there exists $N_0(\lambda_0,\delta)$ such that
$$
\P_b[\bar v^\delta\  {\rm compatible}]\le e^{\tilde CN} 
\widetilde\P[\Delta_d\bar v^\delta ]
$$
for $N>N_0.$
\end{cor}
\dem{Proof of Corollary \ref{cdl.cor.auxmeas}:} We suppose that Lemma \ref{cdl.lem.prob3} holds. Then
\begin{eqnarray*}
P_b[\bar v^\delta\  {\rm compatible}]&\le& e^{N \hat C}
e^{-\lambda_1\sum_{-N+1}^{N-1} |\Delta_d\bar v^\delta(i)+\bar F|}
\\ &=& e^{N \hat C}\left(Z^2\right)^{N}Z^{-1}
\frac{1}{Z^{2N-1}}e^{-\lambda_1\sum_{-N+1}^{N-1} |\Delta_d\bar v^\delta(i)+\bar F|}\le
 e^{N \tilde C}\widetilde\P[\Delta_d\bar v^\delta ]
\end{eqnarray*} for $N$ sufficiently large. Here we can choose e.g. 
$$
\tilde C=2\hat C +2\ln(Z).
$$

\dem{Proof of Lemma \ref{cdl.lem.prob3}:}

In order to simplify notation we write
$$
S_{\bar v^\delta}(\omega)(i):= \tilde S_{M(\bar v^\delta)}(\o)(i,\bar v^\delta(i)).
$$
We write the absolute value as sum of positive and negative part.

By Lemma \ref{cdl.lem.esti5}  we get that there exist  
universal positive constants $C_0$ 
such that the fixed discrete path $\bar v^\delta$ 
is compatible only if
\begin{eqnarray*}&&\omega\in\left(\bigcap_{i=-N+1}^{N-2} 
\left(A_{\bar v^\delta,+}(i)\cap A_{\bar v^\delta,-}(i)\right)\right)\cap A_{\bar v^\delta,+}(N-1)\cap A_{\bar v^\delta,-}(-N+1)\\
 A_{\bar v^\delta,+}(i)&:=&\left\{\omega:\ 
C_0\left(\Delta_d\bar v^\delta(i)+\bar F\right)_+ 
\le S_{\bar v^\delta}(\omega)(i)\right\}\\
 A_{\bar v^\delta,-}(i)&:=&\left\{\omega:\ 
C_0\left(\Delta_d\bar v^\delta(i+1)+\bar F\right)_- 
\le S_{\bar v^\delta}(\omega)(i)\right\}\  \\
B_{\bar v^\delta}(i)&:=&
\left(A_{\bar v^\delta,+}(i)\cap A_{\bar v^\delta,-}(i)\right) .
\end{eqnarray*}
 Note that
$$
B_{\bar v^\delta}(i)\subseteq \left\{
S_{\bar v^\delta}(i)\ge 
\frac{C_0}{2}\left(\Delta_d\bar v^\delta (i+1)+\bar F \right)_-
+ \frac{C_0}{2}\left(\Delta_d\bar v^\delta(i)+\bar F \right)_+\right\}
$$
and we estimate with the help of Corollary \ref{corprob} 
for $i\in \{-N+1,\ldots ,N-2\}$
$$
\P(B_{\bar v^\delta}(i))\le  e^{\hat C-\frac{\widehat\lambda_1 \delta}{C_0}\left(\left(
\Delta_d\bar v^\delta(i)+\bar F \right)_++\left(
\Delta_d\bar v^\delta(i+1)+\bar F \right)_-\right)}$$
for constants $\hat C$ and $\widehat\lambda_1$ depending only on $\lambda_0$
but not on $F.$ 

Moreover,
for $i=N-1$ we obtain
$$
\P(A_{\bar v^\delta,+}(N-1))\le  e^{\hat C-\frac{\widehat\lambda_1
\delta}{C_0}
\left(
\Delta_d\bar v^\delta(N-1)+\bar F 
\right) _+}
$$
and for $i=-N+1$ we obtain 
$$
\P(A_{\bar v^\delta,-}(-N+1))\le  e^{\hat C-\frac{\widehat\lambda_1
\delta}{C_0}
\left(
\Delta_d\bar v^\delta(-N+1)+\bar F
\right) _-}
$$
%$\left(
%\Delta_d\bar v^\delta(i)+\bar F \right)_-=0.$
%(Recall Remark \ref{remdiscrete} (1))

The events $B_{\bar v^\delta}(i)$ 
are independent for different
$i,$ hence
\begin{eqnarray*}
\P_b[\bar v^\delta\  {\rm compatible}]&\le& \P(A_{\bar v^\delta,-}(-N+1))
\P(A_{\bar v^\delta,+}(N-1)) \prod\limits_{i=-N+1}^{N-2} 
\P(B_{\bar v^\delta}(i))
\\ &\le&e^{N \hat C}
e^{-\frac{\widehat \lambda_1 \delta}{C_0}\sum_{-N+1}^{N-1} |\Delta_d\bar v^\delta(i)+\bar F|}.
\end{eqnarray*}
The claim follows now by choosing $\lambda_1=\frac{\widehat\lambda_1 \delta}{C_0}.$

\fdem
\begin{remark}Note that the 1-1-correspondence between second derivatives and paths with
Dirichlet boundary conditions allows us to express each path uniquely through its discrete
Laplacians and thus estimate its probability with the help of the previous lemma.
\end{remark}
As a consequence the discrete Laplacians on average much larger than $-F$ are extremely
unlikely. We will show that nonnegative paths that  cross the "triangle $KN-K|x|$  require such unlikely values
of the discrete Laplacian.

%%%%%%%%%%%%%%%%%%%%%%%%%%%%%%%%%%%%%%%%%%%%%
\section{Final Argumentation}
\subsection{Some formulas on discrete path and Comparison of two paths}
In this section, we recall some well known formulas for discrete 
paths and their discrete derivatives.
The proofs are straightforward computations and therefore omitted.

Let us first recall some basic formulas satisfied by  a discrete path $z$ defined in $\Z\times \R$. 

\begin{lem} \label{cdl.lem.formula}
Let us denote 
$\nl z[\ell+1]:=z[\ell+1]-z[\ell]$ and $\nr z[\ell+1]:=z[\ell+1]-z[\ell+2]$. Then for $\ell\in \Z$ we have 
\begin{itemize}
\item[(i)] \begin{align*}
& \nl z[\ell+1]=\Delta_d z[\ell] +\nl z[\ell]=\sum_{i=1}^{\ell}\Delta_d z[i] +\nl z[1]\\
& \nl z[\ell+1]=\Delta_d z[\ell] +\nl z[\ell]=\sum_{i=k}^{\ell}\Delta_d z[i] +\nl z[k].\end{align*}  

\item[(ii)] \begin{align*}
&z[\ell+1]-z[0]=\sum_{i=1}^{\ell}\sum_{j=1}^{i}\Delta_d z[j] +(\ell+1)\nl z[1].\\
&z[\ell+1]-z[k]=\sum_{i=k+1}^{\ell+1}(z[i]-z[i-1])=\sum_{i=k+1}^{\ell}\sum_{j=k+1}^{i}\Delta_d z[j]  +(\ell+1-k)\nl z[k+1].
\end{align*}  
\item[(iii)] \begin{align*}
&\nr z[0]=\Delta_d z[1] +\nr z[1]=\sum_{i=1}^{\ell }\Delta_d z[i] +\nr z[\ell ],\\
& \nr z[k]=\Delta_d z[k+1] +\nr z[k+1]=\sum_{i=k+1}^{\ell}\Delta_d z[i] +\nr z[\ell].
\end{align*} 
\item[(iv)]  \begin{align*}
&z[0]-z[\ell+1]=\sum_{i=0}^{\ell-1}\sum_{j=i+1}^{\ell}\Delta_d z[j] +(\ell+1)\nr z[\ell]\\
&z[k]-z[\ell+1]=\sum_{i=k}^{\ell} (z[i]-z[i+1])=\sum_{i=k}^{\ell-1}\sum_{j=i+1}^{\ell}\Delta_d z[j] +(\ell+1-k)\nr z[\ell].
\end{align*}
%\item[(v)] Let $N_1<N_2$ be two integers, then we have
 %\begin{align*}
 %&z[N_2]-z[N_1]=\sum_{i=1}^{N_2-N_1}\sum_{j=1}^{i}\Delta_d z[N_1+j] + (N_2-N_1)\nl z[N_1+1]\\ 
%&z[N_1]-z[N_2]=\sum_{i=0}^{N_2-N_1-1}\sum_{j=i}^{N_2-N_1-1}\Delta_d z[N_1+j] +(N_2-N_1)\nr z[N_2-1]. \end{align*}
\item[(v)] $$ \nl z[\ell+1]=-\nr z[\ell]$$
\end{itemize}
\end{lem}
%The proof of (i),(ii),(iv),(v)  follows from obvious computations. 
%\comment{\cg{Chek the formulas!!!}}
\bigskip

Let us now define what we mean by "crossing."
\begin{definition} 
Let $z_1$ and $z_2$ be two given paths in $\Z\times \R$. We say that $z_1$ cross $z_2$  if and only if there exists $i\in \Z$ such that   $z_1[i]\ge z_2[i] $ and $z_1[i+1]\le z_2[i+1]$.
\end{definition}

%With the above definition and the relations (i)--(v) we can compare two  prescribed paths. In particular, we get an %expression involving sums over discrete Laplacians.
%Namely, we have 

%\begin{lem}\label{cdl.lem.crossing}
%Let $z_1$ and $z_2$ be two  paths. Assume that $z_1$ crosses $z_2$ and $z_2$ crosses $z_1$ only once.  Then, %we have for some integer $N_1<N_2$
%\begin{align}
%&z_2[N_2+1]-z_2[N_1+1]-(N_2-N_1)\nl z_2[N_1+1]\le\sum_{i=N_1+2}^{N_2+1}
%\sum_{j=N_1+1}^{i-1}\Delta_d z_1[j] %\label{cdl.eq.crossa}\\
%&z_2[N_1]-z_2[N_2]-(N_2-N_1)\nr z_2[N_2] \le \sum_{i=N_1}^{N_2-1}\sum_{j=i+1}^{N_2}\Delta_d z_1[j]%\label{cdl.eq.crossb}  
%\end{align}
%\end{lem}

We will apply this to the discrete path $\bar v^\delta$ and the triangle $z_K(i):=NK-K|i|.$

\subsection{Proof of Theorem \ref{mainthm}}
First we state  a trivial fact for discrete sums.
\begin{lem}
Let $a_j$ be nonnegative numbers, then 
\begin{equation}\label{averaging}
\sum_{i=1}^N\sum_{j=i}^Na_j=\sum_{i=1}^N j a_j\le N \sum_{i=1}^N a_j\
\end{equation}
\end{lem}
We will show that paths that remain nonnegative but cross the triangle $z_k$ require values
of the average discrete Laplacian which are very unlikely under $\widetilde P.$ In order to do so, we
distinguish cases: Either the path is above the triangle near one of the two endpoints of the interval
$[-N,N]$ and crosses at the interior, or it crosses at $N$ or $-N.$ In both cases, this implies information on the
gradient.  Note that the nonnegativity of the original subsolution does 
not imply the nonnegativity of 
the discretized path, but only that the discretized path is larger than $-\delta FN,$ $\delta$ times
the minimal possible gradient. In particular, it implies that the
terminal value
$b$ of the discretized path is in $[-\delta FN,0].$

{\bf Notation:} As only discrete paths appear in the following estimates,  we will write $v[i]$ for $\bar v^\delta[i]$ 
 In order to simplify notation.

If $-\grad^r v[-N]\le K,$ then by Lemma \ref{cdl.lem.formula}
$$
v[0]-v[-N]=\sum_{i=-N+1}^{-1}\sum_{j=-N+1}^{i}\Delta_d v[j] -N\nr v[-N].$$
Since $v[-N]=0$ and rewriting the double sum the right way, it follows that 
$$
-FN\le v[0]\le NK+\sum_{i=-N+1}^{-1}(-i)(\Delta_d v[i]).
$$
After adding and subtracting $\bar F$ in each term in the summation 
$$
-FN\le  NK+\sum_{i=-N+1}^{-1}(-i)(\Delta_d v[i]+\bar F) -\bar F \frac{N(N-1)}{2}.
$$
 so, invoking (\ref{averaging}) it follows that
$$
\bar F\frac{N(N-1)}{2}-(F+K)N\le 2(N-1)\sum_{i=-N+1}^{N-1}(\Delta_d v[i]+\bar F)_+.
$$
By definition of $\bar F$, we have $$ \bar F\ge  F(1-2(\delta+\eps))-(1+2\delta).$$
Therefore for $\eps$ small, says $ \eps \le \delta$ and $F$ such that   $F\ge 2\frac{1+2\delta}{1-8\delta}$  we achieve $$\bar F\ge \frac{F}{2}.$$ 
 Whence 

$$
F\frac{N(N-1)}{4}-(F+K)N\le 2(N-1)\sum_{i=-N+1}^{N-1}(\Delta_d v[i]+\bar F)_+.
$$
This implies that
for $N$ large and $K$ fixed     
$$
\frac{1}{2(N-1)}\sum_{i=-N+1}^{N-1}(\Delta_d v[i]+\bar F)_+ \ge \frac{1-2\delta}{8}F.
$$ 
As the $(\Delta_d v[i]+\bar F)_+$ are independent  random variables under the auxiliary
probability measure $\widetilde \P$ defined in Def. \ref{Ptilde} which have exponential moments 
bounded as in (\ref{cdl.eq.L=1}),  we can derive an upper bound
for the large deviations principle:
(For the basic form of the large deviations principle needed, see
e.g \cite{Grimmett} Ch. 5.11)
Let
$$
{\mathcal I}(F)=\frac{F}{\mu}-1+\ln\left(\frac{\mu}{F}\right),
$$where $\mu:=\lambda_0^{-1}$ with $\lambda_0$ as in Lemma \ref{cdl.lem.prob1}. (I.e. for exponential random variables $\mu$ is the expectation of $(\Delta_d v[i]+\bar F)_+$ under $\widetilde \P.$ Note that  $\mu$ 
is decreasing in $\lambda_0.$)
Then, by the large deviations principle, 
for any $\eta>0$ there exists $N_0\in \N$ such that for all $N\ge N_0$
$$
\widetilde \P\left( \frac{1}{2(N-1)}\sum_{i=-N+1}^{N-1}(\Delta_d v[i]+\bar F)_+ 
\ge \frac{(1-2\delta)}{8}F\right)\le e^{-N\left(C+{\mathcal I}\left(\frac{(1-2\delta)F}{8}\right)-\eta\right)}.
$$where $C$ is the constant in the bound (\ref{cdl.eq.L=1}). ($C=1$ for exponential random variables.)
Now choose 
%first $\delta$ sufficiently small or $\lambda_0$ sufficiently large 
%and then 
$F$ sufficiently large such that
$$
e^{\widetilde C C-{\mathcal I}\left(\frac{(1-2\delta)F}{8}\right)}<1,
$$
%\comment{ I guess, you mean $e^{\widetilde C-{\mathcal I(\frac{(1-2\delta)F}{8})}/2}<1$}
where the constants are defined in Lemma \ref{cdl.lem.prob3}.

Then  there exists a constant $C_3$ depending on 
$\lambda_0$ and $\delta$ such that for $N$ sufficiently large
$$
\P({\rm case\ 1})\le e^{-C_3N}.
$$

The case $\grad^lv[N]\ge -K$ is done in a similar way.

Second case:
$-\grad^r v[-N]>K,   \grad^lv[N]<-K.$  This implies that the path has to cross the triangle
inside the interval $[-N,N].$
Suppose the path crosses $z_K$ on $[-N,0],$ the other
case is follows by symmetry. Then there exists $N_1, \, -N<N_1<0,$ such that $-\grad ^rv[N_1]\le K$ and
$v[N_1]\le KN.$ Then by Lemma \ref{cdl.lem.formula}
$$
v[N]-v[N_1]=\sum_{i=N_1+1}^{N-1}\sum_{j=N_1+1}^{i}\Delta_d v[j] -(N-N_1)\nr v[N_1],
$$
so
$$
-FN\le v[N]\le 2KN+KN+\sum_{i=N_1+1}^{N-1}\sum_{j=N_1+1}^{i}(\Delta_d v[j]+\bar F) -\bar F \frac{(N-N_1)(N-N_1-1)}{2},
$$
which implies
$$
\bar F \frac{N(N-1)}{2}-(F+3K)N\le \sum_{i=N_1+1}^{N-1}\sum_{j=N_1+1}^{i}(\Delta_d v[j]+\bar F)
\le 2(N-1)\sum_{i=-N+1}^{N-1}(\Delta_d (v[i])+\bar F)_+,
$$i.e. for $N$ sufficiently large
$$
\frac{1}{2(N-1)}\sum_{i=-N+1}^{N-1}(\Delta_d v[i]+\bar F)_+ \ge \frac{(1-2\delta)F}{4}.
$$
Now we can repeat the probabilistic argument from the first case.

Finally, we sum over all possible $-FN$values of the the terminal
condition
$b.$ This sum grows linearly in $N,$ hence using the exponential decay
of the probabilities we obtain
 that there exists 
$C_4(\delta,\lambda_0)$ and $F_0(\delta,\lambda_0)$ such that for
$F>F_0$ 
$$
\P\big(\omega:\ \bar v^\delta 
\mbox{\  compatible\ and\  }
\bar v^\delta\ {\rm crosses}\ z_K  \big)\le e^{-C_4N}.
$$

Now we conclude with Lemma \ref{comparisonlemma}.
 
\fdem

\subsection{Proof of Corollary}

Define $v^N$ as the  solution of the initial-boundary value problem
\begin{eqnarray*}
\frac{\partial v^N}{\partial t}&=& v^N_{xx}(x,t) + \tilde f(x,v^N(x,t)) +F \quad \text{ in } \quad (-N+\delta,N-\delta), \\
v^N(-N,t)&=&u=v^N(N,t)=0\\
v^N(x,0)&=&0,
\end{eqnarray*}
and let $u(x,t)$ solve \ref{cdl.eq.interface}. The comparison principle for parabolic equations implies that
 $v^N(x,t)\le u(x,t)$ for  $x\in [-N-\delta,N+\delta],\ t>0.$  
 
 Moreover, $v^N(x,t)\nearrow v^N_{\rm stat}(x)$ as $t \to \infty ,$ where
 $v_{\rm stat}^N(x)$ is a stationary solution of the Dirichlet problem. 
 
 Note that $\partial_tv^N(x,t)\ge 0$ as
 $\partial_tv^N(x,0)\ge 0,$ and the time derivative $w:=\partial_tv^N$ solves 
$$
\partial_t w=\Delta w+V(x)w,
$$where the potential $V(x)=\frac{\partial f}{\partial u}(x,v^N(x,t))$
is bounded on compact subsets of $\R^N.$ (Note that $\omega$ is a
fixed
parameter here. $v^N\le FN^21_{[-N,N]},$so only obstacles within 
$[-N,N]\times FN^2$ can occur, but these are bounded for $\omega$ fixed.) 

Now a  linear parabolic PDE with sufficiently regular potential
$V(x)$
and nonnegative initial condition remains nonnegative: $\tilde w=
e^{-t\|V\|_\infty}w$ solves 
$$
\partial_t \tilde w=\Delta \tilde w+\tilde V(x)w,\quad \tilde V\le 0
$$ with initial condition $\tilde w\ge 0.$ So the classical parabolic
comparison principle (\cite{Nirenberg}) implies $\tilde w \ge 0.$
 
 By Thm. \ref{mainthm} and the first Borel-Cantelli Lemma (see e.g.
\cite{Grimmett}), 
 $$\P\left(\omega: v_{\rm stat}^N(0)\le 
KN\ {\rm for\  infinitely\ many\ }N\right)=0,$$ so there exist almost surely 
 arbitrarily large $N$ such that
 $$
\liminf_{t\to \infty}  u(0,t)\ge \lim_{t\to\infty} v^N(0,t)=v^N_{\rm
stat}(0)\ge KN,
 $$which implies
\begin{equation}\label{nopin}
 \liminf_{t\to \infty} u(0,t,\omega)=+\infty
\end{equation} with probability 1.
 By the comparison principle, this contradicts the existence
 of a global nonnegative stationary solution.

Moreover, by arguments as in Lemma \ref{comparisonlemma},
(\eqref{nopin}) 
holds for $x\in [-1,1].$ 
As the distribution of the obstacles is invariant
under translations in $x$-direction,  (\eqref{nopin}) holds
for $x\in \R.$
%\bibliographystyle{plain}

%\bibliography{.bib}

\end{document}